\documentclass[12pt]{amsart}
\usepackage{ amsmath, amsthm, amsfonts, amssymb, color}
 \usepackage{mathrsfs}
\usepackage{amsfonts, amsmath}
 \usepackage{amsmath,amstext,amsthm,amssymb,amsxtra}
 \usepackage{txfonts} 
 \usepackage[colorlinks, citecolor=blue,pagebackref,hypertexnames=false]{hyperref}
 \allowdisplaybreaks
 \usepackage{pgf,tikz}
\begin{document}
\baselineskip 16pt

\newcommand\C{{\mathbb C}}
\newtheorem{theorem}{Theorem}[section]
\newtheorem{proposition}[theorem]{Proposition}
\newtheorem{lemma}[theorem]{Lemma}
\newtheorem{corollary}[theorem]{Corollary}
\newtheorem{remark}[theorem]{Remark}
\newtheorem{example}[theorem]{Example}
\newtheorem{question}[theorem]{Question}
\newtheorem{exercise}[theorem]{Exercise}
\newtheorem{definition}[theorem]{Definition}
\newtheorem{conjecture}[theorem]{Conjecture}
\newcommand\RR{\mathbb{R}}
\newcommand{\la}{\lambda}
\def\RN {\mathbb{R}^n}
\newcommand{\norm}[1]{\left\Vert#1\right\Vert}
\newcommand{\abs}[1]{\left\vert#1\right\vert}
\newcommand{\set}[1]{\left\{#1\right\}}
\newcommand{\Real}{\mathbb{R}}
\newcommand{\supp}{\operatorname{supp}}
\newcommand{\card}{\operatorname{card}}
\renewcommand{\L}{\mathcal{L}}
\renewcommand{\P}{\mathcal{P}}
\newcommand{\T}{\mathcal{T}}
\newcommand{\A}{\mathbb{A}}
\newcommand{\K}{\mathcal{K}}
\renewcommand{\S}{\mathcal{S}}
\newcommand{\blue}[1]{\textcolor{blue}{#1}}
\newcommand{\red}[1]{\textcolor{red}{#1}}
\newcommand{\Id}{\operatorname{I}}
\newcommand\wrt{\,{\rm d}}
\def\SH{\sqrt {H}}

\newcommand{\rn}{\mathbb R^n}
\newcommand{\de}{\delta}
\newcommand{\tf}{\tfrac}
\newcommand{\ep}{\epsilon}
\newcommand{\vp}{\varphi}

\newcommand{\mar}[1]{{\marginpar{\sffamily{\scriptsize
        #1}}}}

\newcommand\CC{\mathbb{C}}
\newcommand\NN{\mathbb{N}}
\newcommand\ZZ{\mathbb{Z}}
\renewcommand\Re{\operatorname{Re}}
\renewcommand\Im{\operatorname{Im}}
\newcommand{\mc}{\mathcal}
\newcommand\D{\mathcal{D}}
\newcommand{\al}{\alpha}
\newcommand{\nf}{\infty}
\newcommand{\comment}[1]{\vskip.3cm
	\fbox{%
		\color{red}
		\parbox{0.93\linewidth}{\footnotesize #1}}
	\vskip.3cm}

\newcommand{\disappear}[1]

\numberwithin{equation}{section}
\newcommand{\chg}[1]{{\color{red}{#1}}}
\newcommand{\note}[1]{{\color{green}{#1}}}
\newcommand{\later}[1]{{\color{blue}{#1}}}
\newcommand{\bchi}{ {\chi}}

\numberwithin{equation}{section}
\newcommand\relphantom[1]{\mathrel{\phantom{#1}}}
\newcommand\ve{\varepsilon}  \newcommand\tve{t_{\varepsilon}}
\newcommand\vf{\varphi}      \newcommand\yvf{y_{\varphi}}
\newcommand\bfE{\mathbf{E}}
\newcommand{\ale}{\text{a.e. }}

 \newcommand{\mm}{\mathbf m}
\newcommand{\Be}{\begin{equation}}
\newcommand{\Ee}{\end{equation}}

 \textwidth =162mm
\textheight =228mm
\oddsidemargin=-0.0cm
\evensidemargin=0.0cm
\headheight=13pt
\headsep=0.8cm
\parskip=0pt
\hfuzz=6pt
\widowpenalty=10000
\setlength{\topmargin}{-0.6cm}

\title[The commutators of Bochner-Riesz Operators  ]
{The commutators of Bochner-Riesz Operators for Hermite operator}
\author[Peng Chen, Xixi Lin, and Lixin Yan]{Peng Chen,  Xixi Lin, and Lixin Yan}
\address{Peng Chen, Department of Mathematics, Sun Yat-sen
	University, Guangzhou, 510275, P.R. China}
\email{chenpeng3@mail.sysu.edu.cn}
  \address{Xixi Lin, Department of Mathematics, Sun Yat-sen   University,
	Guangzhou, 510275, P.R. China}
\email{linxx58@mail2.sysu.edu.cn}
\address{Lixin Yan, Department of Mathematics, Sun Yat-sen   University,
	Guangzhou, 510275, P.R. China}
\email{mcsylx@mail.sysu.edu.cn}

\date{\today}

\begin{abstract}
 In this paper, we study  the $L^p$-boundedness of  the commutator
	$
	[b, S_R^\delta(H)](f) = bS_R^\delta(H) f - S_R^\delta(H)(bf)
	$
	of a BMO function  $b$ and the Bochner-Riesz means
	$S_R^\delta(H)$ for Hermite operator $H=-\Delta +|x|^2$ on $\mathbb{R}^n$, $n\geq2$.
	We show that if $\delta>\delta(q)=n(1/q -1/2)- 1/2$, the commutator $[b,S_R^\delta(H)]$ is bounded on $L^p(\mathbb{R}^n)$
	whenever $q<p<q'$  and $1\leq q\leq 2n/ (n+2)$.
\end{abstract}

\maketitle

\section{Introduction}
\setcounter{equation}{0}

We begin with recalling the Bochner-Riesz means on
$\RN$ which  are  defined by,  for $\delta\ge 0$ and $R>0$,
\begin{eqnarray*}
	\widehat{S^{\delta}_Rf}(\xi)
	=\left(1-{|\xi|^2\over R^2}\right)_+^{\delta} \widehat{f}(\xi),  \quad \forall{\xi \in \RN}.
\end{eqnarray*}
Here $(x)_+=\max\{0,x\}$ for $x\in \mathbb R$ and $\widehat{f}\,$ denotes the Fourier  transform
of $f$.  The commutator $[b, S^{\delta}_R]$ is defined by
\begin{eqnarray*}
	[b, S^{\delta}_R](f) = bS^{\delta}_R f - S^{\delta}_R(bf).
\end{eqnarray*}
Let $b\in {\rm BMO}(\RR^n)$. If $\delta\geq(n-1)/2$, the commutator $[b,S^{\delta}_R]$ is bounded on $L^p(\RR^n)$   for all $1<p<\infty$ (\cite[p.129]{LY}); On the space $\RR^2$, if $0<\delta< 1/2$,  Hu and Lu \cite{HL1} and Lu and Xia~\cite{LX} proved that the commutator  $[b,S^{\delta}_R]$ is bounded on $L^p(\RR^2)$   if and only if
$
4/(3+2\delta)<p<4/(1-2\delta).
$
On the space $\RR^n,n\geq3$, if  $0<\delta< (n-1)/2$,  Lu and Xia~\cite{LX} showed that the commutator $[b,S^{\delta}_R]$ is bounded on $L^p(\RR^n)$ with $p>1$,
then
$
\delta>\delta(p)=\max\{n|1/p-1/2|-1/2,0\}.
$
Conversely,   if
$(n-1)/(2n+2)<\delta< (n-1)/2 $
 Hu and Lu \cite{HL1} showed that  the commutator $[b,S^{\delta}_R]$ is bounded on $L^p(\RR^n)$ provided $
\delta>\delta(p)
$.
For other works about the commutator $[b,S_R^{\delta}]$
of Bochner-Riesz operators, see \cite{ABKP, HL1, LX,  LY, SSun, St} and the references therein.

In this paper we  are concerned the  $L^p$-boundedness of commutators of  a BMO function  $b$ and  the    Bochner-Riesz means
for
the Hermite operator $H$  on $ \RR^n$ for $n\geq2$, which is defined by
\begin{eqnarray}\label{e1.1}
	H=-\Delta + |x|^2 =-\sum_{i=1}^n {\partial^2\over \partial x_i^2} + |x|^2, \quad x=(x_1, \cdots, x_n).
\end{eqnarray}
The operator $H$ is non-negative and self-adjoint with respect to the Lebesgue measure
on $\RN$.
Let   $\Phi_{\mu}$ be  eigenfunctions for
the Hermite operator with eigenvalue $(2|\mu|+n)$ and $\{\Phi_{\mu}\}_{\mu\in \mathbb N_0^n}$ form a complete orthonormal
system in $L^2({\RN})$.
Thus, for  every $f\in L^2(\RN)$  we have  the Hermite expansion
\begin{eqnarray} \label{e1.2}
	f(x)=\sum_{\mu}\langle f, \Phi_{\mu}\rangle \Phi_\mu(x)=\sum_{k=0}^{\infty}  \sum_{|\mu|=k}\langle f, \Phi_{\mu}\rangle\Phi_\mu(x)
	=:  \sum_{k=0}^{\infty} P_kf(x).
\end{eqnarray}
For $R>0$, the  Bochner-Riesz means for $H$ of order $\delta\geq 0$   is defined by
\begin{eqnarray}\label{e13}
	S_R^{\delta}(H)f(x):=\Big(I-\frac{H}{R^2}\Big)^{\delta}_+
	:=
	\sum_{k=0}^{\infty} \left(1-{2k+n\over R^2}\right)_+^{\delta} P_k f(x).
\end{eqnarray}
On the space $\mathbb{R}$,
Thangavelu \cite{T1} showed that $S_R^\delta(H)$ is uniformly bounded on $L^p(\mathbb{R})$ for $1\leq p\leq \infty$
provided $\delta>1/6$. If $0<\delta<1/6$,  $S_R^\delta(H)$ is uniformly bounded on $L^p(\mathbb{R})$ if and only if
$4/(6\delta+3) <p<4/(1-6\delta)$. On the space $\mathbb{R}^n$, $n\geq2$, if $\delta> (n-1)/2$, Thangavelu \cite{T2}
showed that $S_R^\delta(H)$ is uniformly bounded on $L^p(\mathbb{R}^n)$ for $1\leq p\leq \infty$. Especially for the case
$p=1$, $\delta>(n-1)/2$ is also the necessary condition for the $L^1(\mathbb{R}^n)$-boundedness of $S_R^\delta(H)$.
When $0\leq\delta\leq (n-1)/2$, it was conjectured (see \cite[p.259]{T3}) that $S_R^\delta(H)$ is bounded on $L^p(\mathbb{R}^n)$ uniformly in $R$ if and only if
$$\delta>\delta(p)=\max\Bigg\{n\left|\frac1p-\frac12\right|-\frac12,0\Bigg\}.$$
Thangavelu  proved the necessary part that if $S_R^\delta(H)$  is $L^p(\mathbb{R}^n)$ uniformly bounded,
then $\delta>\delta(p)$.  Karadzhov \cite{K} proved the sufficiency of the conjecture is true when $\delta$ is in the range of $[1,2n/(n+2)]\cup[2n/(n-2),\infty]$.

Let $b\in {\rm BMO}(\mathbb R^n)$.  The commutator $[b, S_R^\delta(H)]$ is defined by
$
[b, S_R^\delta(H)](f) = bS_R^\delta(H) f - S_R^\delta(H)(bf).
$
The aim of this  paper is to  prove the following result.

\begin{theorem}\label{Thm1}
	Let $H$ be the Hermite operator defined in \eqref{e1.1} on $\RR^n$, $n\geq2$. Let $1\leq p\leq2n/(n+2)$ and $\delta>\delta(p)$, then for all $b\in \mathrm{BMO(\mathbb{R}^n)}$
	$$\sup_{R>0} \big\| \big[b,S_R^{\delta}(H)\big ]  \big\|_{q\rightarrow q}\leq C\|b\|_{\mathrm{BMO}}$$
	for all $p<q<p'$.
\end{theorem}

We would like to mention that in  \cite{CTW},    Tian,  Ward and the first named author obtained
the $L^q(\mathbb{R}^n)$-boundedness of  the commutator  $[b, S_R^\delta(L)]$
of a BMO function $b$ and the Bochner-Riesz means
$S_R^\delta(L)$ for a class  of elliptic self-adjoint operators $L$   on $\mathbb{R}^n$,
where $L$ satisfies  the finite speed of propagation property  for the wave equation (see section 2 below)
and spectral measure estimate: For some $1\leq p<2,$
\begin{equation} \label{e1.4}
	\|dE_{\sqrt L}(\lambda)\|_{p\to  p'  }\le C\ \lambda^{{n}(\frac{1}{p}-\frac{1}{p'})-1}, \ \ \lambda>0.
\end{equation}
However,  such estimate \eqref{e1.4} is not available for the Hermite operator $H=-\Delta +|x|^2 $ on ${\mathbb R^n},$ and    an argument   in \cite{CTW} does not work    Theorem~\ref{Thm1} very well.
 The proof of   Theorem~\ref{Thm1} will be given  in Section  3 by using   three non-trivial facts.
The first is the restriction-type estimate  due to
Karadzhov \cite{K}:    for $1\leq p\leq  2n/(n+2)$ and $n\geq2$,
\begin{align}\label{SC}
	\|\chi_{[k,k+1)}(H)\|_{p\rightarrow2}=\|P_k\|_{p\rightarrow2} \leq Ck^{\frac{n}2(\frac1p-\frac12)-\frac12},\ \forall k\in\mathbb{N}.
\end{align}
The discrete restriction-type
condition~\eqref{SC} is weaker than the restriction-type condition~\eqref{e1.4}. To compensate for this
difference, when proving  the $L^p$-boundedness for the commutator
$\big[b, 	S_R^{\delta}(H)\big ]$ in our  Theorem~\ref{Thm1}, we need
\emph{a priori} estimate that for any $\varepsilon>0$ and $1\leq r\leq 2$,
\begin{align}\label{Cruc}
	\|(I+H)^{-\frac{n}{2}(\frac1r-\frac12)-\varepsilon}\|_{2\rightarrow r}\leq C_\varepsilon.
\end{align}
This is an important observation
in our paper. Further,  the fact that the
eigenvalue of the Hermite operator is bigger  than $1$ or equal $1$,
 is useful in the proof of Theorem~\ref{Thm1}.

\medskip

\section{ Preliminaries }
\setcounter{equation}{0}

  Throughout this paper,
 for $1\leq p\leq\infty$, we write the $L^p$ norm of a function $f$ by $\|f\|_p$ and write $\|T\|_{p\to q} $ for the operator norm of $T$ if $T$ is a bounded linear operator from $L^p(\mathbb{R}^n)$ to $L^q(\mathbb{R}^n)$. Given a
 subset $E\subseteq X$, we  denote the characteristic function of $E$ by $\chi_E$.
For a given  function $F: {\mathbb R}\to {\mathbb C}$ and $R>0$, we define the
function
$\delta_RF:  {\mathbb R}\to {\mathbb C}$ by putting
 $\delta_RF(x)= F(Rx).$

\begin{lemma}\label{Lem2}
	Let $H$ be the Hermite operator defined in \eqref{e1.1}. For any $\varepsilon>0$ and $1\leq r\leq2$, there exists a constant $C=C_{r,\,\varepsilon}$ such that
  \begin{align}\label{lemma2.1}
    \|(I+H)^{-\frac{n}2(\frac1r-\frac12)-\varepsilon}\|_{2\rightarrow r}\leq C.
  \end{align}
\end{lemma}
\begin{proof}
 The proof of \eqref{lemma2.1}   is proved in \cite[Lemma 7.9 ]{DOS} in the case   $r=1$.
 For general $ 1\leq r \leq 2$,
 we refer to \cite[p.3823 (6.5)]{CLSY}.
\end{proof}

\begin{lemma}\label{le2.2}
	Let $H$ be the Hermite operator defined in \eqref{e1.1} on $\RR^n$, $n\geq2$.  Then  for $1\leq p\leq  2n/(n+2)$,
\begin{align}\label{SC1}
	\|\chi_{[k,k+1)}(H)\|_{p\rightarrow2}=\|P_k\|_{p\rightarrow2} \leq Ck^{\frac{n}2(\frac1p-\frac12)-\frac12}
=Ck^{(\delta(p)-1/2)/2},\ \  \ \forall k\in\mathbb{N}.
\end{align}
\end{lemma}

\begin{proof}
	For the proof of this lemma, we refer the reader to \cite{K}. See also \cite{CLSY, KoT, T3}.
	\end{proof}

Define
$$\mathfrak{D}_t:=\{(x,y)\in\mathbb{R}^n\times\mathbb{R}^n:|x-y|\leq t\}.$$
Given an operator $T$ from $L^p$ to $L^q$, we denote
$$\mathrm{supp}K_T\subseteq \mathfrak{D}_t,$$
if $\langle Tf_1,f_2 \rangle=0$ whenever $f_i$ has support $B(x_i,t_i),i=1,2$ and $t_1+t_2+t<|x_1-x_2|$. If $T$ is an integral operator,
 then $\mathrm{supp}K_T\subseteq \mathfrak{D}_t$ coincides with $K_T(x,y)=0,\forall(x,y)\notin\mathfrak{D}_t$. We say that an operator $L$
  satisfies {\it finite speed propagation property}, it means that
\begin{align}\label{FS}
	\mathrm{supp}K_{\cos(t\sqrt{L})}\subseteq \mathfrak{D}_t,\ \forall t>0.
	\tag{FS}
\end{align}
\begin{lemma}\label{le2.3}
Let 	$F$ be  an even bounded Borel function and $\widehat{F}\in L^1(\mathbb{R})$ with supp$\widehat{F}\subseteq [-t,t]$.
	Then we have that $K_{F(\sqrt{H})}\subseteq \mathfrak{D}_t$.
\end{lemma}

\begin{proof}
Since the heat kernel $p_t(x,y)$ of $e^{-tH}$ satisfies Gaussian upper bound, i.e.,
$$p_t(x,y)\leq (4\pi t)^{-n/2}\exp\left(-\frac{|x-y|^2}{4t}\right),$$
it follows from  (see for example, \cite[Theorem 2]{S}) that $H$ satisfies finite speed propagation property \eqref{FS}.
By Fourier inversion, for any  even function  $F$,
$$F(\sqrt{H})=\frac{1}{2\pi}\int_{-\infty}^{+\infty}\widehat{F}(t)\cos(t\sqrt{H})dt.$$
It's known from \cite[Lemma I.1]{COSY} with supp$\widehat{F}\subseteq [-t,t]$,
we have that $
K_{F(\sqrt{H})}\subseteq \mathfrak{D}_t.
$
\end{proof}

\medskip

\section{Proof of Theorem~\ref{Thm1}}
\setcounter{equation}{0}

We start with the following lemma.

\begin{lemma}\label{le3.1}
	Suppose that $T$ is a linear map and $T$ is bounded from $L^p$ to $L^s$ for some $1\leq p<s<\infty$. Assume also that
	$$\mathrm{supp}~K_T\subseteq \mathfrak{D}_{\rho}$$
	for some $\rho>0$. Assume that function $b\in\mathrm{BMO}$. Then given a number $q$ with $1\leq p<q<s$,
	there exists a constant $C=C_{p,q,s}$ such that
	$$\|[b,T]\|_{q\rightarrow q}\leq C\|b\|_{\mathrm{BMO}}\rho^{n(\frac1p-\frac1s)}\|T\|_{p\rightarrow s}.$$
\end{lemma}

\begin{proof}
For the proof, we refer the reader to \cite[Lemma 3.2]{CTW}.	
\end{proof}

 To prove Theorem~\ref{Thm1},  we will show the following result, which gives Theorem~\ref{Thm1}
 as a special case with $F(\lambda)=(1-\lambda^2)^{\delta}_+$ and $t=R^{-1}$.

\begin{theorem}\label{Thm2}
  Let $H$ be the Hermite operator defined in \eqref{e1.1} on $\RR^n$, $n\geq2$, $1\leq p\leq 2n/(n+2)$ and $s>n(1/p-1/2)$. Suppose $b\in \mathrm{BMO(\mathbb{R}^n)}$. Then for any even Borel function
  such that supp$F\subseteq[-1,1]$ and $F\in W^{2}_{s}(\mathbb{R})$,  the commutator $[b,F(t\sqrt{H})]$
   is bounded on $L^q(\mathbb{R}^n)$ for all $p<q<p'$ uniformly in $t$. In addition,
 \begin{eqnarray}\label{ee3.1}
 	\sup_{t>0}\left\|[b,F(t\sqrt{H})]\right\|_{q\rightarrow q}\leq C\|b\|_{\mathrm{BMO}}\|F\|_{W^{2}_{s}}.
 \end{eqnarray}
\end{theorem}

Before we start the proof of Theorem~\ref{Thm2}, let us show that Theorem~\ref{Thm1} is a straighforward consequence of Theorem~\ref{Thm2}. Indeed,
we take  $F(\lambda)=(1-\lambda^2)^{\delta}_+$ and $t=R^{-1}$. Note that $F\in W^2_{s}$ if and only if $\delta>s-1/2$.
	Hence, it follows from Theorem~\ref{Thm2} that for all $\delta>\delta(p)=n(1/p-1/2)-1/2$ and $p<q<p',$
$$
	\|[b, F(t\sqrt{H})]\|_{q\rightarrow q}\leq  C\|b\|_{\rm BMO}\|F\|_{W^2_{s}}\leq C\|b\|_{\rm BMO}
$$
	with a bound $C$   independent of $t$. Then we have
	$$\sup_{R>0}\big\|\big[b, S_R^{\delta}(H) \big]\big\|_{q\rightarrow q}\leq C\|b\|_{\mathrm{BMO}}.$$ 	
	This completes the proof of Theorem \ref{Thm1}.

\medskip

The proof of Theorem~\ref{Thm2}  is inspired by \cite{COSY,CTW,DOS}.
To  prove Theorem~\ref{Thm2}, we
fix an even function   $\eta\in C_c^{\infty}(\mathbb{R}^n)$  supported in    $  \{ 1/4 \leq |u|\leq 1\}$
such that   $\sum_{\ell \in\mathbb{Z}}\eta(2^{-\ell }u)=1, \forall u\neq0$. Set $\eta_0(u)=1-\sum_{\ell \geq1} \eta(2^{-\ell}u).$
Write  $\eta_{\ell }(u)=\eta(2^{-\ell}u)$ for $\ell\geq 1$.  Define
  \begin{align}\label{aaab2}
    F^{(\ell)}(\lambda)=\frac{1}{2\pi} \int_{-\infty}^{+\infty} \eta_{\ell}(u)\widehat{F}(u)e^{i\lambda u}du,\ \ \ell\geq0.
  \end{align}
Then we have
\begin{align}\label{aaab3}
  F(t\sqrt{H})=\sum_{\ell=0}^{\infty}F^{(\ell)}(t\sqrt{H}).
\end{align}
Noting that supp\,$F\subseteq[-1,1]$ and the eigenvalue of the Hermite operator is not less than $1$,
  we can always assume that $0<t\leq1$ in estimate~\eqref{ee3.1}. Let $N=50([t^{-1}]+1)(N\approx t^{-1})$. We pick up an even function $\xi\in C_c^{\infty}([-1,1])$ with
\begin{align}\label{ke1}
\widehat{\xi}^{(0)}(0)=1,\ \ \widehat{\xi}^{(1)}(0)=\cdots=\widehat{\xi}^{(\kappa)}(0)=0,
\end{align}
where $\kappa$ is a large enough integer  that will be chosen later. Let  $\xi_N(\lambda)=N\xi(N\lambda)$.
 Then
    \begin{align*}
     [b,F(t\sqrt{H})]&=      [b, \delta_tF(\sqrt{H})]\nonumber\\
     &= [b,(\xi_N*\delta_tF)(\sqrt{H})]+[b,(\delta_tF-\xi_N*\delta_tF)(\sqrt{H})].
\end{align*}

It  is  clear that to prove Theorem                                                                                                                                                                                                                                                                                                           ~\ref{Thm2}, it is sufficient to show the following two lemmas~\ref{le3.3} and  \ref{le3.4}.

\begin{lemma}\label{le3.3}
Suppose $b\in \mathrm{BMO(\mathbb{R}^n)}$. Let $1\leq p\leq 2n/(n+2)$, $s>n(1/p-1/2)$ and $p<q<p'$.
 Then for   any even Borel function such that supp$F\subseteq[-1,1]$ and $F\in W^{2}_{s}(\mathbb{R})$,
	\begin{eqnarray}\label{e3.1}
		\sup_{t>0}\left\|[b,(\xi_N*\delta_tF)(\sqrt{H})]\right\|_{q\rightarrow q}\leq C\|b\|_{\mathrm{BMO}}\|F\|_{W^{2}_{s}}.
	\end{eqnarray}
\end{lemma}

\smallskip

\begin{lemma}\label{le3.4}
	Suppose $b\in \mathrm{BMO(\mathbb{R}^n)}$. Let $1\leq p\leq 2n/(n+2)$, $s>n(1/p-1/2)$ and $p<q<p'$.
 Then for   any  even Borel function
  such that supp$F\subseteq[-1,1]$ and $F\in W^{2}_{s}(\mathbb{R})$,
	\begin{eqnarray}\label{e3.2}
		\sup_{t>0}\left\| [b,(\delta_tF-\xi_N*\delta_tF)(\sqrt{H})]\right\|_{q\rightarrow q}\leq C\|b\|_{\mathrm{BMO}}\|F\|_{W^{2}_{s}}.
	\end{eqnarray}
\end{lemma}

\medskip

\subsection{Proof of Lemma~\ref{le3.3}}
  First, we note that
 $$\supp \widehat{(\xi_N*\delta_tF^{(\ell)}})\subseteq[-2^{\ell+1}t,2^{\ell+1}t].
 $$
 By Lemma~\ref{le2.3}, we have that  $K_{\xi_N*\delta_tF^{(\ell)}}\subseteq\mathfrak{D}_{2^{\ell+1}t}$.
 We then apply Lemma~\ref{le3.1}  to see that for $p<q<2$,
  \begin{align}\label{eefe1}
  \|[b,(\xi_N*\delta_tF)(\sqrt{H})]\|_{q\rightarrow q}
  \leq C\|b\|_{\mathrm{BMO}}\sum_{\ell=0}^{\infty}\|(\xi_N*\delta_tF^{(\ell)})(\sqrt{H})\|_{p\rightarrow {2}}(2^{\ell}t)^{n(\frac1p-\frac1{2})},
  \end{align}
and we will estimate each term $\|\xi_N*\delta_tF^{(\ell)}(\sqrt{H})\|_{p\rightarrow {2}}$ in the sequel. To do it,
 we pick up a function  $\psi\in C_c^{\infty}(\mathbb{R})$ with support $[-4,4]$ and $\psi=1$ in $[-2,2]$ to write
\begin{align}\label{e3.7}
 (\xi_N*\delta_tF^{(\ell)})(\sqrt{H}) =\psi(t\sqrt{H})(\xi_N*\delta_tF^{(\ell)})(\sqrt{H})
+ (1-\psi)(t\sqrt{H})(\xi_N*\delta_tF^{(\ell)})(\sqrt{H}).
 \end{align}
 Recall that $1\leq p\leq 2n/(n+2), n\geq2$, we have $\delta(p)=n(1/p-1/2)-1/2$. It follows by Lemma~\ref{le2.2}  that
  \begin{align}\label{meq1}
\|\psi(t\sqrt{H})(\xi_N*\delta_tF^{(\ell)})(\sqrt{H})f\|_{2}^2
&= \sum_{k\geq0}\psi^2(t\sqrt{2k+n})(\xi_N*\delta_tF^{(\ell)})^2(\sqrt{2k+n})\|P_kf\|_2^2\nonumber\\
&\leq C\sum_{k\geq0}k^{\delta(p)-1/2}\psi^2(t\sqrt{2k+n})(\xi_N*\delta_tF^{(\ell)})^2(\sqrt{2k+n})\|f\|_p^2.
 \end{align}
Noting supp\,$\psi\subseteq[-4,4]$, supp\,$\xi\subseteq[-1,1]$, $t\approx N^{-1}$ and $\delta(p)\geq1/2$, we have
 \begin{align}\label{meq2}
\|\psi(t\sqrt{H})(\xi_N*\delta_tF^{(\ell)})(\sqrt{H})\|_{p\rightarrow2}^2
&\leq C\sum_{t\sqrt{2k+n}\leq4}k^{\delta(p)-1/2}(\xi_N*\delta_tF^{(\ell)})^2(\sqrt{2k+n})\nonumber\\
&\leq Ct^{-2(\delta(p)-1/2)}\sum_{t\sqrt{2k+n}\leq4}(\xi_N*\delta_tF^{(\ell)})^2(\sqrt{2k+n})\nonumber\\
&\leq C\|\xi\|_2^2t^{-2(\delta(p)-1/2)}\sum_{t\sqrt{2k+n}\leq4}\int_{N\sqrt{2k+n}-1}^{N\sqrt{2k+n}+1}|F^{(\ell)}(tN^{-1}y)|^2dy\nonumber\\
&\leq Ct^{-2(\delta(p)-1/2)}\int_{\mathbb{R}}|F^{(\ell)}(tN^{-1}y)|^2dy\approx Ct^{-2n(\frac1p-\frac12)}\|F^{(\ell)}\|_2^2,
 \end{align}
where  in the last inequality we use the fact that the sets $\big\{[N\sqrt{2k+n}-1,N\sqrt{2k+n}+1]\big\}_k$ are disjoint whenever $t\sqrt{2k+n}\leq4$ .

Now we consider the term $(1-\psi)(t\sqrt{H})(\xi_N*\delta_tF^{(\ell)})(\sqrt{H})$.
We apply Lemma~\ref{le2.2} to obtain
  \begin{align}\label{meq3}
\|(1-\psi)(t\sqrt{H})(\xi_N*\delta_tF^{(\ell)})(\sqrt{H})\|_{p\rightarrow2}^2
&\leq C\sum_{k\geq0}k^{\delta(p)-1/2}(1-\psi)^2(t\sqrt{2k+n})(\xi_N*\delta_tF^{(\ell)})^2(\sqrt{2k+n})\nonumber\\
&\leq C\sum_{t\sqrt{2k+n}\geq2}k^{\delta(p)-1/2}(\xi_N*\delta_tF^{(\ell)})^2(\sqrt{2k+n}).
\end{align}
Since  supp\,$\xi\in[-1,1]$, supp\,$(1-\psi)\in[2,\infty)$, supp\,$F\in[-1,1]$,  $\widehat{\eta_{\ell}}$ is a Schwartz function and $0<tN^{-1}\leq1/2$,  we have that for any $M>0$
\begin{align}\label{ee11}
|(1-\psi)(t\lambda)(\xi_N*\delta_tF^{(\ell)})(\lambda)|
&\leq \|\xi\|_1
|(1-\psi)(t\lambda)|\|\chi_{[t\lambda-tN^{-1},t\lambda+tN^{-1}]}F*\widehat{\eta_{\ell}}\|_{L^{\infty}}\nonumber\\
 &\leq C_{M}\chi_{|t\lambda|\geq2}\int_{-1}^{1}|F(y)|2^{\ell}\left(1+2^{\ell}\big||t\lambda|-y\big|\right)^{-M}dy\nonumber\\[4pt]
&\leq C_{M}\chi_{|t\lambda|\geq2}2^{-(M-1)\ell}|t\lambda|^{-M}\|F\|_2.
 \end{align}
Hence,
\begin{align}\label{meq4}
  {\rm RHS \ of }\ \eqref{meq3}&\leq C2^{-2(M-1)\ell}\|F\|_2^2\sum_{t\sqrt{2k+n}\geq2}k^{\delta(p)-1/2}(t\sqrt{2k+n})^{-2M}\nonumber\\
  &\leq C2^{-2(M-1)\ell}t^{-2n(\frac1p-\frac12)}\|F\|_2^2.
\end{align}
Recall that  $s>n(1/p-1/2)$, $W^2_s(\mathbb{R})\subseteq B^{2,1}_{n(\frac1p-\frac12)}(\mathbb{R})$,
which in combination with estimates~\eqref{meq2} and \eqref{meq4} implies that
  \begin{align*}
 {\rm RHS \ of }\ \eqref{eefe1}\leq &C_{M} \|b\|_{\mathrm{BMO}}
 \sum_{\ell\geq0}\left(2^{\ell n(\frac1p-\frac1{2})}\|F^{(\ell)}\|_2+2^{-(M-1)\ell}\|F\|_2\right)\\
 \leq&C\|b\|_{\mathrm{BMO}}\left(\|F\|_{B^{2,1}_{n(\frac1p-\frac12)}}+\|F\|_2\right)
 \leq C\|b\|_{\mathrm{BMO}}\|F\|_{W^2_s}.
  \end{align*}
By duality and interpolation argument, we see that for $p<q< p'$ and $s>n(1/p-1/2)$,
$$\|[b,(\xi_N*\delta_tF)(\sqrt{H})]\|_{q\rightarrow q}\leq C\|b\|_{\mathrm{BMO}}\|F\|_{W_s^2}.$$
This completes the proof of Lemma~\ref{le3.3}.
   \hfill{}$\Box$

\medskip

\subsection{Proof of Lemma~\ref{le3.4}}  To prove Lemma~\ref{le3.4}, we need the following result.

\begin{lemma}\label{le3.5}
 Let  $1\leq p\leq 2n/(n+2)$. Suppose $\theta\in C_c^{\infty}([-8,8])$.   With the notations as in Lemma~\ref{le3.4}, there exists a constant $C$ independent of $t$ such that
\begin{align}\label{Pro3.5}
&\Big\|\theta(2^{-j}t\sqrt{H})(\delta_tF^{(\ell)}-\xi_N*\delta_tF^{(\ell)})(\sqrt{H})\Big\|_{p\rightarrow2}\nonumber\\
&\leq
\begin{cases}
C\|\theta\|_{\infty}\|F\|_{W^2_s}2^{-j/2}N^{-2s}(2^{j}N)^{n(\frac1p-\frac12)}(2^{\ell}N^{-2})^{\kappa+1-s}, & \mathrm{if}\   2^{\ell}tN^{-1}\leq1;  \\[6pt]
C\|\theta\|_{\infty}\|F\|_{W^2_s}2^{-j/2}N^{-2s}(2^{j}N)^{n(\frac1p-\frac12)}(2^{\ell}N^{-2})^{1/2-s}, & \mathrm{if} \  2^{\ell}tN^{-1}>1
\end{cases}
  \end{align}
for all $j,\ell\in\mathbb{N}$, where $\kappa$ is the constant in estimate~\eqref{ke1}.
\end{lemma}

The proof of Lemma~\ref{le3.5} will be given later.

\medskip

Now let us apply Lemma~\ref{le3.5} to prove Lemma~\ref{le3.4}.
We pick up functions $\psi\in C_c^{\infty}([-4,4])$ and $\phi\in C_c^{\infty}([2,8])$ such that
$
  \psi(\lambda)+\sum_{j\geq0} \phi(2^{-j}\lambda)=1 $  for all   $ \lambda>0.
$
Let $p<r<2$. For any $0<\varepsilon<1/2$, we  consider $r_1$ such that $r<r_1\leq2$ and $n(1/p-1/{r_1})<\varepsilon/2$.
Then we apply Lemma~\ref{le3.1} with \eqref{aaab3} to obtain
   \begin{align}\label{epe1}
 &\hspace{-0.5cm} \| [b,(\delta_tF-\xi_N*\delta_tF)(\sqrt{H})]\|_{r\rightarrow r}\nonumber\\
 &\leq C\|b\|_{\mathrm{BMO}}\sum_{\ell=0}^{\infty}
 \|  (\delta_tF^{(\ell)}-\xi_N*\delta_tF^{(\ell)})(\sqrt{H}) \|_{p\rightarrow {r_1}}(2^{\ell}t)^{n(\frac1p-\frac1{r_1})}
  \end{align}
and
\begin{align}\label{ffeq1}
 \|(\delta_tF^{(\ell)}-\xi_N*\delta_tF^{(\ell)})(\sqrt{H})\|_{p\rightarrow {r_1}}
&\leq  \|\psi(t\sqrt{H})(\delta_tF^{(\ell)}-\xi_N*\delta_tF^{(\ell)})(\sqrt{H})\|_{p\rightarrow {r_1}}\nonumber\\
&+\sum_{j\geq0} \| \phi(2^{-j}t\sqrt{H})(\delta_tF^{(\ell)}-\xi_N*\delta_tF^{(\ell)})(\sqrt{H})\|_{p\rightarrow {r_1}}\nonumber\\
&=: E_1(\ell, t)  +\sum_{j\geq0} E_2(j,\ell,  t).
\end{align}
In the following, we set  $\gamma=n(1/{r_1}-1/2)$.  For the term $E_1(\ell, t)$,  we note that
 supp~$\psi\subseteq[-4,4]$ and $N\approx t^{-1}> 1$. This  gives
 \begin{align*}
 \sup_{\lambda} |\psi(t\lambda)(1+|\lambda|^2)^{(\gamma+\varepsilon)/2}|\leq Ct^{-(\gamma+\varepsilon)}\leq CN^{\gamma+\varepsilon}.
 \end{align*}
 By Lemma \ref{Lem2} and Lemma~\ref{le3.5},
  \begin{align*}
E_1(\ell, t)
&\leq\|\psi(t\sqrt{H})(I+H)^{(\gamma+\varepsilon)/2}(\delta_tF^{(\ell)}-\xi_N*\delta_tF^{(\ell)})(\sqrt{H})\|_{p\rightarrow2}
\|(I+H)^{-(\gamma+\varepsilon)/2}\|_{2\rightarrow r_1}\nonumber\\
&\leq C\|\psi(t\sqrt{H})(I+H)^{(\gamma+\varepsilon)/2}(\delta_tF^{(\ell)}-\xi_N*\delta_tF^{(\ell)})(\sqrt{H})\|_{p\rightarrow2}\nonumber\\
 &
 \leq
  \begin{cases}
    C\|F\|_{W^2_s}N^{n(\frac1p-\frac12)+\gamma-2s+\varepsilon}(2^{\ell}N^{-2})^{\kappa+1-s}, & \mbox{if } 2^{\ell}tN^{-1}\leq1; \\[4pt]
   C\|F\|_{W^2_s}N^{n(\frac1p-\frac12)+\gamma-2s+\varepsilon}(2^{\ell}N^{-2})^{1/2-s}, & \mbox{if } 2^{\ell}tN^{-1}\geq1.
  \end{cases}
  \end{align*}
This, in combination  with the fact that $s>n(1/p-1/2)\geq1$, $n(1/p-1/{r_1})<\varepsilon/2$ and our selection of $\kappa$ such that $\kappa\geq s$,  yields
  \begin{align}\label{ffeq2}
  &\sum_{\ell=0}^{\infty}\|b\|_{\mathrm{BMO}}\|\psi(t\sqrt{H})(\delta_tF^{(\ell)}-\xi_N*\delta_tF^{(\ell)})(\sqrt{H})
  \|_{p\rightarrow {r_1}}(2^{\ell}t)^{n(\frac1p-\frac1{r_1})}\nonumber\\
  &\leq C\|b\|_{\mathrm{BMO}}\|F\|_{W^2_s}N^{2n(\frac{1}{r_1}-\frac12)-2s+\varepsilon}\nonumber\\
  &\hspace{0.5cm} \times\left(\sum_{2^{\ell}tN^{-1}<1}2^{\ell n(\frac1p-\frac1{r_1})}(2^{\ell}N^{-2})^{\kappa+1-s}+\sum_{2^{\ell}tN^{-1}
  \geq1}2^{\ell n(\frac1p-\frac1{r_1})}(2^{\ell}N^{-2})^{1/2-s}\right)\nonumber\\
  &\leq  C\|b\|_{\mathrm{BMO}}\|F\|_{W^2_s}N^{2n(\frac1p-\frac12)-2s+\varepsilon}.
  \end{align}

  Now we estimate the terms $E_2(j,\ell,  t)$. Since supp~$\phi\subseteq[2,8]\subseteq[-8,8]$ and $N\approx t^{-1}>1$, we have that
   $\sup_{\lambda} |\phi(2^{-j}t\lambda)(1+|\lambda|^2)^{(\gamma+\varepsilon)/2}|\leq C(2^jN)^{\gamma+\varepsilon}.$
By Lemma \ref{Lem2} and Lemma~\ref{le3.5},
\begin{align}\label{epe3}
E_2(j,\ell,  t)
 &\leq C\|\phi(2^{-j}t\sqrt{H})(I+H)^{(\gamma+\varepsilon)/2}(\delta_tF^{(\ell)}-\xi_N*\delta_tF^{(\ell)})(\sqrt{H})\|_{p\rightarrow2} \nonumber\\
 &\leq
 \begin{cases}
  C\|F\|_{W^2_s}2^{-j/2}N^{-2s}(2^jN)^{n(\frac1p-\frac12)+\gamma+\varepsilon}(2^{\ell}N^{-2})^{\kappa+1-s}, & \mbox{if } 2^{\ell}tN^{-1}\leq1; \\[4pt]
  C\|F\|_{W^2_s}2^{-j/2}N^{-2s}(2^jN)^{n(\frac1p-\frac12)+\gamma+\varepsilon}(2^{\ell}N^{-2})^{1/2-s}, & \mbox{if } 2^{\ell}tN^{-1}\geq1.
 \end{cases}
\end{align}
On another hand, we also have that
\begin{align}\label{epep2}
E_2(j,\ell,  t)&\leq C\left(\sum_{k\geq0}k^{\delta(p)-1/2}(2k+n)^{\gamma+\varepsilon}\phi^2(2^{-j}t\sqrt{2k+n})
(\delta_tF^{(\ell)}-\xi_N*\delta_tF^{(\ell)})^2(\sqrt{2k+n})\right)^{1/2}\nonumber\\
&\leq  C2^{-M\ell}2^{-(M+1)j}\|F\|_1\left(\sum_{2^{j+1}\leq t\sqrt{2k+n}\leq2^{j+3}}(2k+n)^{\delta(p)-1/2+\gamma+\varepsilon}\right)^{1/2}\nonumber\\
 &\leq C 2^{-M\ell}2^{-(M+1)j}(2^jN)^{n(\frac1p-\frac12)+\gamma+\varepsilon}\|F\|_2,
 \end{align}
where in the second inequality we use the inequality that
\begin{align*}
	|\phi(2^{-j}t\lambda)|\left(|\delta_tF^{(\ell)}(\lambda)|+|\xi_N*\delta_tF^{(\ell)}(\lambda)|\right)
	&\leq C_{M}2^{-M\ell}2^{-(M+1)j}\|F\|_2,
\end{align*}
which can be obtained by a similar discussion of estimate~\eqref{ee11}.

Next we set
$\alpha:=n(1/p-1/2)+\gamma+\varepsilon$, $\beta:=n(1/p-1/{r_1})$ and $\nu=(\alpha-1/2)/(M+1/2)$
with $ M>\alpha-1$. Recall that $1\leq p\leq 2n/(n+2)$, so $\alpha>1$ and $0<\nu<1$.
When  $2^{\ell}N^{-1}t\leq1$, we choose  $j_1>0$   such that
$$(2^{j_1}N)^{\alpha}2^{-j_1/2}N^{-2s}(2^{\ell}N^{-2})^{\kappa+1-s}=(2^{j_1}N)^{\alpha}2^{-(M+1)j_1}2^{-M\ell}.$$
 This, in combination with  estimates~\eqref{epe3} and  \eqref{epep2},   yields
\begin{align}\label{e3.19}
\sum_{j\geq0} E_2(j,\ell,  t)
\leq& C_M\|F\|_{W_s^2}\left(\sum_{j\geq j_1}2^{-M\ell}(2^jN)^{\alpha}2^{-(M+1)j}
+\sum_{j\leq j_1}2^{-j/2}(2^jN)^{\alpha}N^{-2s}(2^{\ell}N^{-2})^{\kappa+1-s}\right)\nonumber\\
\leq&C_M\|F\|_{W_s^2}2^{-M\nu \ell}N^{\alpha-2s(1-\nu)}(2^{\ell}N^{-2})^{(1-\nu)(\kappa+1-s)}.
\end{align}
When  $2^{\ell}N^{-1}t\geq1$, we choose  $j_2>0$   such that
 $$(2^{j_2}N)^{\alpha}2^{-j_2/2}N^{-2s}(2^{\ell}N^{-2})^{1/2-s}=(2^{j_2}N)^{\alpha}2^{-(M+1)j_2}2^{-M\ell}.$$
It follows by estimates~\eqref{epe3} and  \eqref{epep2} that
\begin{align}\label{e3.20}
\sum_{j\geq0}E_2(j,\ell,  t)
\leq& C_M\|F\|_{W_s^2}\left(\sum_{j\geq j_2}2^{-M\ell}(2^jN)^{\alpha}2^{-(M+1)j}
+\sum_{j\leq j_2}2^{-j/2}(2^jN)^{\alpha}N^{-2s}(2^{\ell}N^{-2})^{1/2-s}\right)\nonumber\\
\leq&C_M\|F\|_{W_s^2}2^{-M\nu \ell}N^{\alpha-2s(1-\nu)}(2^{\ell}N^{-2})^{(1-\nu)(1/2-s)}.
\end{align}
Combining estimates~\eqref{ffeq1}, \eqref{e3.19} and \eqref{e3.20}, we  use the facts that $s> n(1/p-1/2)\geq1$, $\kappa\geq s$ and $N\approx t^{-1}$ to see that
\begin{align}\label{ffeq3}
&\sum_{\ell\geq0}\sum_{j\geq0}\|b\|_{\mathrm{BMO}}(2^{\ell}t)^{\beta}\|\phi(2^{-j}t\sqrt{H})
(\delta_tF^{(\ell)}-\xi_N*\delta_tF^{(\ell)})(\sqrt{H})\|_{p\rightarrow r_1}\nonumber\\
&\leq C_M\|b\|_{\mathrm{BMO}}\|F\|_{W_s^2}N^{\alpha-2s(1-\nu)}\nonumber\\
&\hspace{0.5cm}  \times\left(\sum_{2^{\ell}N^{-1}t\leq1}2^{-M\nu \ell}(2^{\ell}N^{-2})^{(1-\nu)(\kappa+1-s)}(2^{\ell}t)^{\beta}
+\sum_{2^{\ell}N^{-1}t\geq1}2^{-M\nu \ell}(2^{\ell}N^{-2})^{(1-\nu)(1/2-s)}(2^{\ell}t)^{\beta}\right)\nonumber\\
 &\leq C_M\|b\|_{\mathrm{BMO}}\|F\|_{W_s^2}N^{\alpha+\beta-2M\nu-2s(1-\nu)}\nonumber\\
 &\hspace{0.5cm} \times \left(\sum_{2^{\ell}N^{-1}t\leq1}(2^{\ell}N^{-2})^{-M\nu+(1-\nu)(\kappa+1-s)+\beta}
 +\sum_{2^{\ell}N^{-1}t\geq1}(2^{\ell}N^{-2})^{-M\nu+(1-\nu)(1/2-s)+\beta}\right)\nonumber\\
 &\leq C_M\|b\|_{\mathrm{BMO}}\|F\|_{W_s^2}N^{\alpha+\beta-2s-2(M-s)\nu}\nonumber\\
 &\leq C\|b\|_{\mathrm{BMO}} \|F\|_{W_s^2},
\end{align}
 where we take $M$ and $\kappa$ such that $M\geq \max\{s,2\alpha\}$ and $\kappa\geq 2\alpha+s-2>s$. Recall that $\alpha>1$, $\beta<\varepsilon/2<1/4$ and $0<\nu<1$. When $s>\max\{(\alpha+\beta)/2,1/2\}$, i.e. $s>n(1/p-1/2)+\varepsilon/2$, we have
\begin{eqnarray}\label{eff1}
\left\{
\begin{array} {ll}
&\alpha+\beta-2s-2(M-s)\nu<0,\\[5pt]
& -M\nu+(1-\nu)(\kappa+1-s)+\beta>0,\\[5pt]
& -M\nu+(1-\nu)(1/2-s)+\beta<0.
\end{array}
\right.
\end{eqnarray}
So estimate~\eqref{ffeq3} holds.

Combining the estimates~\eqref{epe1}, \eqref{ffeq1}, \eqref{ffeq2}, \eqref{ffeq3}, we obtain that for any $s>n(1/p-1/2)+\varepsilon$ and any $r$ such that
 $0<1/p-1/r<\varepsilon/2$,
$$\|[b,(\delta_tF-\xi_N*\delta_tF)(\sqrt{H})]\|_{r\rightarrow r}\leq C_{\varepsilon}\|b\|_{\mathrm{BMO}}\|F\|_{W_s^2}.$$
By duality and interpolation argument, we have  proved estimate~\eqref{e3.2}  of Lemma~\ref{le3.4}
  for $p<q< p'$ and $s>n(1/p-1/2)+\varepsilon$.  Hence, we obtain the proof of  Lemma~\ref{le3.4}
provided Lemma~\ref{le3.5} is proved.
    \hfill{}$\Box$

 \medskip

Finally, let us  prove Lemma~\ref{le3.5}.

\begin{proof}[Proof of Lemma~\ref{le3.5}]
It follows from the Hermite expansion, Lemma~\ref{le2.2} and supp~$\theta\subseteq[-8,8]$ that
\begin{align}\label{feq1}
  &\Big\|\theta(2^{-j}t\sqrt{H})(\delta_tF^{(\ell)}-\xi_N*\delta_tF^{(\ell)})(\sqrt{H})\Big\|_{p\rightarrow2} \nonumber\\
  & \leq C\left(\sum_{k\geq0}k^{\delta(p)-1/2}\theta^2(2^{-j}t\sqrt{2k+n})(\delta_tF^{(\ell)}-\xi_N*\delta_tF^{(\ell)})^2(\sqrt{2k+n})\right)^{1/2}\nonumber\\
  &\leq C\|\theta\|_{\infty}(2^jt^{-1})^{\delta(p)-1/2}
  \left(\sum_{t\sqrt{2k+n}\leq2^{j+3}}(\delta_tF^{(\ell)}-\xi_N*\delta_tF^{(\ell)})^2(\sqrt{2k+n})\right)^{1/2}.
\end{align}
Let $0<\mu\leq1$, we first define two functions $H$ and $\Omega_{\mu,j}^{(\ell)}$ by
\begin{align*}
  &\widehat{H}(\lambda)=|2^{j}\lambda|^s2^jNt^{-1}\widehat{F}(2^{j}Nt^{-1}\lambda),\\
  &\widehat{\Omega}_{\mu,j}^{(\ell)}(\lambda)=|2^{j}\lambda|^{-s}(1-\widehat{\xi}(2^j\lambda))\eta_{\ell}(2^{j}\mu^{-1}\lambda).
\end{align*}
Write $\xi_{2^{-j}}(\lambda)=2^{-j}\xi(2^{-j}\lambda)$. Observe that
\begin{align}\label{feq2}
  (\delta_tF^{(\ell)}-\xi_N*\delta_tF^{(\ell)})(\sqrt{2k+n})
  &=(\delta_{t(2^jN)^{-1}}F^{(\ell)}-\xi_{2^{-j}}*\delta_{t(2^jN)^{-1}}F^{(\ell)})(2^jN\sqrt{2k+n})\nonumber\\
  &=H*\Omega_{tN^{-1},j}^{(\ell)}(2^jN\sqrt{2k+n}).
\end{align}
We write $\Omega_{\mu,j}^{(\ell)}$ as $\sum_{m\in\mathbb{Z}} \Omega_{\mu,j}^{(\ell),m}(\lambda),$
where
\begin{align*}
   \Omega_{\mu,j}^{(\ell),m}(\lambda):=
\begin{cases}
  \    \Omega_{\mu,j}^{(\ell)}(\lambda)\chi_{[m-1,m+1)}(\lambda), & m \ \mathrm{is}\  \mathrm{odd}; \\
  \    0, & m \ \mathrm{is} \ \mathrm{even}.
\end{cases}
\end{align*}
From H\"older's inequality and Minkowski's inequality, we have
\begin{align}\label{ineq4}
&\left(\sum_{t\sqrt{2k+n}\leq2^{j+3}}|H*\Omega_{tN^{-1},j}^{(\ell)}(2^jN\sqrt{2k+n})|^2\right)^{1/2}\nonumber\\
&\leq\left(\sum_{t\sqrt{2k+n}\leq2^{j+3}}\left(\sum_{m\in\mathbb{Z}}\left\|\Omega_{tN^{-1},j}^{(\ell),m}\right\|_2
\left(\int_{2^jN\sqrt{2k+n}-m-1}^{2^jN\sqrt{2k+n}-m+1}|H(y)|^2dy\right)^{1/2}\right)^2\right)^{1/2}\nonumber\\
&\leq\sum_{m\in\mathbb{Z}}\|\Omega_{tN^{-1},j}^{(\ell),m}\|_2
\left(\sum_{t\sqrt{2k+n}\leq2^{j+3}}\int_{2^jN\sqrt{2k+n}-m-1}^{2^jN\sqrt{2k+n}-m+1}|H(y)|^2dy\right)^{1/2}\nonumber\\
&\leq \sum_{m\in\mathbb{Z}}\|\Omega_{tN^{-1},j}^{(\ell),m}\|_2\|H\|_2,
 \end{align}
where  we use the fact that the sets $\big\{[2^jN\sqrt{2k+n}-m-1,2^jN\sqrt{2k+n}-m+1]\big\}_{k,m}$
are disjoint whenever $t\sqrt{2k+n}\leq2^{j+3}$ in the last inequality. Indeed, when $t\sqrt{2k+2+n}\leq 2^{j+3}$, we have
$$\sqrt{2k+n+2}+\sqrt{2k+n}\leq 2^{j+4}t^{-1}\leq 2^jN,$$
which is equivalent to
$$2^jN\sqrt{2k+n}-m+1\leq 2^jN\sqrt{2k+2+n}-m-1.$$
By the definition of homogeneous Sobolev spaces and $t\approx N^{-1}$, we have
\begin{align*}
\|H\|_2&\leq 2^{js}\|\delta_{t(2^jN)^{-1}}F\|_{\dot{W}^2_s}=2^{j/2}(tN^{-1})^{s-1/2} \|F\|_{\dot{W}^2_s}
\leq C2^{j/2}N^{-2s+1}\|F\|_{\dot{W}^2_s},
\end{align*}
which in combination with estimates~\eqref{feq2}-\eqref{ineq4} implies that
\begin{align}\label{eq3}
 {\rm RHS \ of}\ \eqref{feq1}\leq  C\|\theta\|_{\infty}\|F\|_{W^2_s}2^{-j/2}(2^jN)^{n(\frac1p-\frac12)}N^{-2s}\sum_m\|\Omega_{tN^{-1},j}^{(\ell),m}\|_2.
\end{align}
Hence,  the proof of  estimate~\eqref{Pro3.5} reduces  to  show that for any $\ell,j\in\mathbb{N}$ and any $0<\mu\leq1$,
\begin{align}\label{lemma4.3}
  \sum_{m\in\mathbb{Z}}\|\Omega_{\mu,j}^{(\ell),m}\|_2\leq
\begin{cases}
     C(2^{\ell}\mu)^{\kappa+1-s}, &  \mathrm{ if} \ 2^{\ell}\mu\leq1;\\[4pt]
     C(2^{\ell}\mu)^{1/2-s}, & \mathrm{if} \ 2^{\ell}\mu>1,
\end{cases}
\end{align}
since  the desired estimate~\eqref{Pro3.5} follows easily by setting $\mu=tN^{-1}$ and taking estimate~\eqref{lemma4.3}
into estimate~\eqref{eq3}.

Let us  prove estimate~\eqref{lemma4.3}. We write $\Omega_{\mu,j}^{(\ell)}(\lambda):=2^{-j}\zeta_{\mu}^{(\ell)}(2^{-j}\lambda)$, where
$$\widehat{\zeta_{\mu}^{(\ell)}}(a):=(1-\widehat{\xi}(a))\eta_{\ell}(\mu^{-1}a)|a|^{-s}.
$$
 We will show that for any $0<\mu\leq1$,
 \begin{align}\label{ffae1}
  |\zeta_{\mu}^{(0)}(\lambda)|\leq C(1+|\mu \lambda|)^{-(\kappa+2-s)}\mu^{\kappa+2-s},
\end{align}
and for every  $\ell\in\mathbb{N}^+$ and $M\geq 0$,
\begin{align}\label{lemma3.2}
  |\zeta_{\mu}^{(\ell)}(\lambda)|\leq
  \begin{cases}
    C_M(2^{\ell}\mu)^{\kappa+2-s}(1+|2^{\ell}\mu \lambda|)^{-M}, & \mathrm{if} \  2^{\ell}\mu\leq1; \\[6pt]
    C_M(2^{\ell}\mu)^{1-s}(1+|2^{\ell}\mu \lambda|)^{-M}, & \mathrm{if}\  2^{\ell}\mu\geq1.
  \end{cases}
\end{align}
Recall that $\widehat{\xi}$ and $\eta_0$ are even functions, we see that
\begin{align*}
  \zeta_{\mu}^{(0)}(\lambda)&=\int_{0}^{\infty}(1-\widehat{\xi}(a))\eta_0(\mu^{-1}a)a^{-s}(e^{ia\lambda}+e^{-ia\lambda})da\\
 &=\sum_{j\leq0}\int_{0}^{\infty}(1-\widehat{\xi}(a))\eta(2^{-j}\mu^{-1}a)a^{-s}(e^{ia\lambda}+e^{-ia\lambda})da\\
 &=\sum_{j\leq0}(2^j\mu)^{1-s}\int_{0}^{\infty}(1-\widehat{\xi}(2^j\mu a))\eta(a)a^{-s}\left(e^{i2^j\mu \lambda a}+e^{-i2^j\mu \lambda a}\right)da.
\end{align*}
 For any $\alpha\in\mathbb{N}$, from Taylor's expansion of $\widehat{\xi}$ at original point and the compact support of $\eta$, we have
 $ |\partial^{\alpha}\left((1-\widehat{\xi}(2^j\mu a))\eta(a)a^{-s}\right) |\leq C_{\alpha}\chi_{\{a:1/4\leq|a|\leq1\}}(a)(2^j\mu)^{\kappa+1}.$
Then for any $M\geq0$,
 $  |\zeta_{\mu}^{(0)}(\lambda)|\leq C_M\sum_{j\leq0}(2^j\mu)^{\kappa+2-s}(1+|2^j\mu\lambda|)^{-M}.$
From these estimates,  we   obtain that for any $0<\mu\leq1$
\begin{align*}
  |\zeta_{\mu}^{(0)}(\lambda)|\leq
    \begin{cases}
      C|\lambda|^{-(\kappa+2-s)}, &  |\lambda\mu|\geq1; \\[4pt]
      C\mu^{\kappa+2-s}, &  |\lambda\mu|\leq1,
    \end{cases}
  \end{align*}
which gives estimate~\eqref{ffae1}.

Now we consider the case  $\ell\geq1$. Recall that $\widehat{\xi}$ and $\eta_l$ are even functions. By integration by parts, we have
\begin{align*}
|\zeta_{\mu}^{(\ell)}(\lambda)|&=|\int_{\mathbb{R}}(1-\widehat{\xi}(a))\eta(2^{-\ell}\mu^{-1}a)|a|^{-s}e^{ia\lambda}da|\\
 &=|\int_{0}^{\infty}(1-\widehat{\xi}(a))\eta(2^{-\ell}\mu^{-1}a)a^{-s}(e^{ia\lambda}+e^{-ia\lambda})da|\\
 &\leq C|\lambda^{-\alpha}|\int_{0}^{\infty}|\partial^{\alpha}\left((1-\widehat{\xi}(a))\eta(2^{-\ell}\mu^{-1}a)a^{-s}\right)|da.
 \end{align*}
 If $2^{\ell}\mu\leq1$, for any $\alpha\in\mathbb{N}$, from supp~$\eta\subseteq\{\tau:1/4\leq|\tau|\leq1\}$,
 Taylor's expansion of $\widehat{\xi}$ at the original point and $\widehat{\xi}^{(0)}(0)=1$, $\widehat{\xi}^{(1)}(0)
 =\cdots=\widehat{\xi}^{(\kappa)}(0)=0$, we have
 $$|\partial^{\alpha}\left((1-\widehat{\xi}(a))\eta(2^{-\ell}\mu^{-1}a)a^{-s}\right)|
 \leq C_{\alpha}\chi_{\{a:2^{\ell-2}\mu\leq|a|\leq2^{\ell}\mu\}}(a)(2^{\ell}\mu)^{\kappa+1-s-\alpha}.$$
If $2^{\ell}\mu>1$, for any $\alpha\in\mathbb{N}$, from supp~$\eta\subseteq\{\tau:1/4\leq|\tau|\leq1\}$ and $\xi\in C_c^{\infty}(\mathbb{R})$, we have
$$|\partial^{\alpha}\left((1-\widehat{\xi}(a))\eta(2^{-\ell}\mu^{-1}a)a^{-s}\right)|\leq C_{\alpha}\chi_{\{a:2^{\ell-2}\mu\leq|a|\leq2^{\ell}\mu\}}(a)(2^{\ell}\mu)^{-s-\alpha}.$$
Consequently, for any $\ell\in\mathbb{N}^+$,  $M\in\mathbb{N}$ and any $0<\mu\leq1$,
\begin{align*}
|\zeta_{\mu}^{(\ell)}(\lambda)|\leq
\begin{cases}
  C_M(2^{\ell}\mu)^{\kappa+2-s}(1+|2^{\ell}\mu\lambda|)^{-M}, & \mbox{if }  2^{\ell}\mu\leq1;\\[6pt]
  C_M(2^{\ell}\mu)^{1-s}(1+|2^{\ell}\mu\lambda|)^{-M}, &  \mbox{if } 2^{\ell}\mu>1,
\end{cases}
\end{align*}
which gives estimate~\eqref{lemma3.2}.
\smallskip

Next we return to apply estimates~\eqref{ffae1} and \eqref{lemma3.2}  to prove estimate~\eqref{lemma4.3}.
Recall that $\Omega_{\mu,j}^{(\ell)}(x):=2^{-j}\zeta_{\mu}^{(\ell)}(2^{-j}x)$, which is an even function.
It suffices to consider  the positive odd integer $m$  since
 $\Omega_{\mu,j}^{(\ell),m}=0$ when $m$ is an even integer.
In the following, we assume that $\ell\in\mathbb{N}^+$  and consider two cases: $2^{\ell}\mu\leq1$ and $2^{\ell}\mu >1$.

\medskip

\noindent
{\bf{Case 1.}\ \   $2^{\ell}\mu\leq1$}.

In this case, it follows from estimate~\eqref{lemma3.2} that for $m\in\mathbb{N}^+$ and for every $M>1,$
 \begin{align*}
   \|\Omega_{\mu,j}^{(\ell),m}\|_2&=\left(\int_{m-1}^{m+1}|\Omega_{\mu,j}^{(\ell)}(x)|^2dx\right)^{1/2}
   =2^{-j}\left(\int_{m-1}^{m+1} |\zeta_{\mu}^{(\ell)}(2^{-j}x)|^2dx\right)^{1/2}\\
   &\leq C_{M}2^{-j}(2^{\ell}\mu)^{\kappa+2-s}\left(\int_{m-1}^{m+1} (1+2^{\ell-j}\mu x)^{-2M}dx\right)^{1/2}\\
   &\leq C_{M}2^{-j}(2^{\ell}\mu)^{\kappa+2-s}\left(1+2^{\ell-j}\mu (m-1)\right)^{-M}.
 \end{align*}
 Hence, for any $j\in\NN$ and $0<\mu\leq1$,
\begin{align*}
  \sum_{m\geq1}\|\Omega_{\mu,j}^{(\ell),m}\|_2\leq&\sum_{\substack{m:m\geq1,
  2^{\ell-j}\mu(m-1)\leq 1}}\|\Omega_{\mu,j}^{(\ell),m}\|_2+\sum_{\substack{m:m\geq1,2^{\ell-j}\mu(m-1)> 1}}\|\Omega_{\mu,j}^{(\ell),m}\|_2\\
  \leq&C_{M}2^{-j}(2^{\ell}\mu)^{\kappa+2-s}\Bigg(\sum_{\substack{m:m\geq1, \\ 2^{\ell-j}\mu(m-1)\leq 1}}1
  +\sum_{\substack{m:m\geq1,\\2^{\ell-j}\mu(m-1)> 1}}(2^{\ell-j}\mu (m-1))^{-M}\Bigg)\\
  \leq&  C2^{-j}(2^{\ell}\mu)^{\kappa+2-s}\left(\#\left\{m\in\mathbb{N}^+:1\leq m\leq 2^j(2^{\ell}\mu)^{-1}+1\right\}+2^j(2^{\ell}\mu)^{-1}\right)\\[4pt]
  \leq &C(2^{\ell}\mu)^{\kappa+1-s},
\end{align*}
where in the last inequality we use the fact   that when $ 2^{\ell}\mu \leq1$, $j\geq0$,
$$\#\{m\in\mathbb{N}^+:1\leq m\leq 2^j(2^{\ell}\mu)^{-1}+1\}\leq 2^{j+1}(2^{\ell}\mu)^{-1}.
$$

\medskip

\noindent
{\bf{Case 2.}\ \   $2^{\ell}\mu >1$}.

In this case, it follows from estimate~\eqref{lemma3.2}  that for $m=1$ and for every $M>1,$
 \begin{align*}
   \|\Omega_{\mu,j}^{(\ell),1}\|_2&=\left(\int_{0}^{2}|\Omega_{\mu,j}^{(\ell)}(x)|^2dx\right)^{1/2}
   =2^{-j}\left(\int_{0}^{2} |\zeta_{\mu}^{(\ell)}(2^{-j}x)|^2dx\right)^{1/2}\\
   &\leq C_{M}2^{-j}(2^{\ell}\mu)^{1-s}\left(\int_{0}^{2} (1+2^{\ell-j}\mu x)^{-2M}dx\right)^{1/2}\\
   &\leq C_{M}2^{-j}(2^{\ell}\mu)^{1-s}(2^{\ell-j}\mu)^{-1/2}\left(\int_{0}^{\infty}(1+y)^{-2M}dy\right)^{1/2}\\
   &\leq C2^{-j/2}(2^{\ell}\mu)^{1/2-s}.
 \end{align*}
If  $m\geq3$, then for any $M>1, $
 \begin{align*}
   \|\Omega_{\mu,j}^{(\ell),m}\|_2&=\left(\int_{m-1}^{m+1}|\Omega_{\mu,j}^{(\ell)}(x)|^2dx\right)^{1/2}
   =2^{-j}\left(\int_{m-1}^{m+1} |\zeta_{\mu}^{(\ell)}(2^{-j}x)|^2dx\right)^{1/2}\\
   &\leq C_{M}2^{-j}(2^{\ell}\mu)^{1-s}\left(\int_{m-1}^{m+1} (1+2^{\ell-j}\mu x)^{-2M}dx\right)^{1/2}\\
   &\leq C_{M}2^{-j}(2^{\ell}\mu)^{1-s}\left(1+2^{\ell-j}\mu (m-1)\right)^{-M}.
 \end{align*}
Therefore, for any $j\in\NN$ and $0<\mu\leq1$,
\begin{align*}
  \sum_{m\geq1}\|\Omega_{\mu,j}^{(\ell),m}\|_2\leq& \|\Omega_{\mu,j}^{(\ell),1}\|_2+\sum_{m:m\geq3,2^{\ell-j}\mu(m-1)\leq 1}
  \|\Omega_{\mu,j}^{(\ell),m}\|_2+\sum_{m:m\geq3,2^{\ell-j}\mu(m-1)> 1}\|\Omega_{\mu,j}^{(\ell),m}\|_2\\[4pt]
  \leq&C_{M}2^{-j}(2^{\ell}\mu)^{1-s}\Bigg(2^{\frac{j}2}(2^\ell\mu)^{-\frac12}+\Bigg(\sum_{\substack{m:m\geq3,\\2^{\ell-j}\mu(m-1)\leq 1}}1
  +\sum_{\substack{m:m\geq3,\\2^{\ell-j}\mu(m-1)> 1}}(2^{\ell-j}\mu (m-1))^{-M}\Bigg)\Bigg)\\
  \leq &C2^{-j}(2^{\ell}\mu)^{1-s}\left(2^{\frac{j}2}(2^\ell\mu)^{-\frac12}+\#\left\{m\in\mathbb{N}^+:3\leq m\leq 2^j(2^{\ell}\mu)^{-1}+1\right\}+2^j(2^{\ell}\mu)^{-1}\right)\\[6pt]
  \leq &C2^{-j}(2^{\ell}\mu)^{1-s}\Big(2^{\frac{j}2}(2^\ell\mu)^{-\frac12}+3\times2^j(2^{\ell}\mu)^{-1}\Big)\leq C(2^{\ell}\mu)^{\frac12-s},
\end{align*}
where  we use the inequality that
$$\#\{m\in\mathbb{N}^+:3\leq m\leq 2^j(2^{\ell}\mu)^{-1}+1\}\leq \max\{0,2^j(2^{\ell}\mu)^{-1}-1\}\leq 2^{j+1}(2^{\ell}\mu)^{-1} .$$

The proof of $\ell=0$ is similar to that of {\bf{Case 1}}   and even more simple,  and we skip the detail here. Combining the above discussion, we obtain estimate~\eqref{lemma4.3}.
This finishes the proof of Lemma~\ref{le3.5}, and then   the proof of Lemma~\ref{le3.4} is complete.
\end{proof}

\medskip

\section{Extensions}
\setcounter{equation}{0}
In the previous section, we proved   Theorem~\ref{Thm1} with the potential $V=|x|^2.$
However, the form of this potential does not play a crucial role in the proof.
Here we consider instead the operators
  $ H_V= - \Delta + V$
on $L^2(\RR^n)$ for $n\geq2$, where $V$ is  a  positive potential with the following conditions:
\begin{equation}\label{e4.1}
V \sim |x|^2, \quad |\nabla V| \sim |x|, \quad |\partial_x^2 V| \le 1.
\end{equation}
Under the assumption \eqref{e4.1}, $ H_V$ is a nonnegative self-adjoint operator on $L^2(\mathbb{R}^n)$. Such an operator admits a spectral
resolution
\begin{eqnarray*}
H_V=\int_0^{\infty} \lambda dE_{H_V}(\lambda).
\end{eqnarray*}
Now, the Bochner-Riesz means of order $\delta\geq 0$ for operator $ H_V$  can be defined by
 \begin{equation}\label{e4.2}
 S^{\delta}_R(H_V) f :=  \int_0^{R^2} \left(1-\frac{\lambda}{R^2}\right)^{\delta}dE_{H_V}(\lambda) f, \ \ \ \ f\in L^2(\mathbb{R}^n).
   \end{equation}
Then, similar to the Hermite operator $H$, the commutator $[b,S^{\delta}_R(H_V)]$ is  bounded on  $L^q(\mathbb{R}^n))$ for $p<q<p'$ uniformly in $R>0$ for $1\leq p\leq2n/(n+2)$ and $\delta>\delta(p)$, as we now show.

 \begin{theorem}\label{th4.1} \, Suppose the potential $V$ satisfies~\eqref{e4.1}.
 Let $n\geq2$, $1\leq p\leq2n/(n+2)$ and $\delta>\delta(p)$, then for all $b\in \mathrm{BMO(\mathbb{R}^n)}$
	$$\sup_{R>0}\big\| \big[b, 	S_R^{\delta}(H_V)\big ]  \big\|_{q\rightarrow q}\leq C\|b\|_{\mathrm{BMO}}$$
	for all $p<q<p'$.
	\end{theorem}

\begin{proof}
It follows from  \cite[Theorem~4]{KoT} that  for all  $1\leq p\leq 2n/(n+2)$ and $\lambda \geq 0$
\begin{eqnarray}\label{e4.4}
\|E_{H_V}[\lambda^2,\, \lambda^2+1)\|_{p\rightarrow 2} \leq C(1+\lambda)^{n(\frac1p-\frac12)-1}.
\end{eqnarray}
With Lemma~\ref{le2.2} and the spectral projection  estimate \eqref{e4.4},  the argument   in the proof of Theorem~\ref{Thm2}
also establishes Theorem~\ref{th4.1}.
\end{proof}

\bigskip

 \noindent
{\bf Acknowledgments.}
P. Chen was supported by NNSF of China 12171489.
 L. Yan was supported by the NNSF of China
 11871480  and by the Australian Research Council (ARC) through the research
grant DP190100970.

%
%
%
 \noindent
{\bf Declarations}

 \noindent
{\bf Conflict of interest} The authors declare that they have no conflict of interest regarding the work reported in this paper.

\end{document}